\documentclass{article}
\usepackage{amsmath}
\usepackage{float}
\usepackage{color}
\usepackage{footnote}
\usepackage{cases}
\usepackage{array}
\usepackage{bm}
\usepackage{enumitem}
\usepackage{cite}
\usepackage[bottom]{footmisc}
\usepackage{graphicx}
\usepackage{pfnote}
\usepackage{appendix}
\usepackage{caption}
\usepackage[labelformat=simple]{subcaption}
\usepackage{setspace}
\usepackage{amsthm}
\usepackage[normalem]{ulem}
\usepackage{afterpage}
\usepackage{soul}
\usepackage{centernot}
\usepackage{parskip}
\usepackage{amsfonts}
\usepackage{placeins}
\usepackage[nottoc]{tocbibind}
\usepackage{titlesec}
\numberwithin{equation}{section}
\newtheorem{definition}{Definition}
\newtheorem{theorem}{Theorem}
\newtheorem{lemma}{Lemma}
\newtheorem{corollary}{Corollary}

\makeatletter

\DeclareCaptionSubType*{figure}

\DeclareCaptionLabelFormat{opening}{Figure~#2:}
\makeatother

\def\inti#1{\int_{x_0}^{x_1}{#1}\mathop{}\!\mathrm{d}x}
\def\intii#1{\int_{I}{#1}\mathop{}\!\mathrm{d}x}
\def\intx#1{\int_X{#1}\!\mathop{}\!\mathrm{d}x}
\def\intia#1{\int_{0}^{d}{#1}\mathop{}\!\mathrm{d}x}
\def\norm#1#2#3{||{#1}||_{L^{#2}(#3)}}

\newcommand{\cmab}{${\cal C}_{1}(m,a,b)$ }

\title{Loss of Regularity in the $K(m,n)$ Equations}
\author{Alon Zilburg and Philip Rosenau\\
   {\em School of Mathematics, Tel Aviv University}\\
  {\em Tel Aviv, Israel 69978}}
\begin{document}
\maketitle
\begin{abstract}
 Using a priori estimates we prove that initially nonnegative, smooth and compactly supported solutions of the $K(m,n)$ equations must lose their smoothness in finite time. 
 Formation of a singularity is a prerequisite for the emergence of compactons.
\end{abstract}
\section{Introduction}\label{sec:intro}
In this work we consider the initial value problem for the $K(m,n)$ equations \cite{Rosenau01}
\begin{equation}\label{kmn}
u_t+(u^m)_x+(u^n)_{xxx}=0,\quad m,n\geq2
\end{equation}
which are the prototypical equations that support the formation and evolution of Compactons\cite{rosenau02}, solitary waves with compact support. Thus, for example, the $K(2,2)$ equation admits the compacton (anti-compacton) solution
\begin{equation}\label{k22c}
u^c(x,t)=\pm\frac{4\lambda}3\cos^2\left(\frac{x\mp\lambda t}{4}\right)H(2\pi-|x\mp\lambda t|)
\end{equation}
where $H(s)=\begin{cases}1, s\geq 0\\0, s<0\end{cases},$ is the Heaviside function and the $(-)$ branch refers to anti-compactons. We call the reader's attention to the jump in the second derivative of the compacton at its fronts.\\ Some of the presented results will be seen to apply to a more general class of nonlinear dispersion equations\cite{Rosenau03}
\begin{equation}\label{C1}
{\cal C}_{1}(m,a,b):~~ u_{t} +(u^m)_{x} + \frac{1}{b}[u^{a}(u^{b})_{xx}]_{x}=0,\quad a\geq0,m,b\geq2,n:=a+b
\end{equation}
which for $a=0$ reduces, up to normalization which plays no role in the following, to the $K(m,n)$ equations.\vspace{5mm}\\
In most of the previous works on $K(m,n)$ equations the focus has mainly been on the dynamics of compactons and their interaction. In contradistinction, we focus on the initial phase of evolution prior to the formation of compactons of an assumed smooth solution. We view this stage of evolution to be crucial for it propels the evolution toward gradients catastrophe which is necessary for the emergence of compactons.\\
The plan of the paper is as follows:\\
In section \ref{sec:estimates} we derive a priori estimates for strong solutions (see Definition 1) of the \cmab equations and prove that they remain confined within their initial support (Lemma \ref{support}). It is exactly the jump  discontinuity at the front of $u^c_{xx}$ which makes it possible for the $K(2,2)$ compacton to propagate.  We show that initially nonnegative (hereafter may be replaced by nonpositive) strong  solutions remain so for the duration of their existence (Lemma \ref{positivity}).\\
 Thereafter we derive an upper bound on the existence time of initially nonnegative and compactly supported strong solutions of the $K(m,n)$ equations (Theorem \ref{finitetime}). This bound depends on the mass $\int u\mathop{}\!\mathrm{d}x$ and the support of the initial excitation. The key point is that the motion of solution's center of mass is independent of the dispersion. Thus the convection forces solution's center of mass to move past any finite point in space which is in conflict with the confinement within the initial support. Existence of an additional conservation law for $K(n+1,n)$ equations enables us to improve the bound  (Corollary \ref{knp1n}). We shall also discuss the possibility of extending our results for additional classes of equations. In section \ref{sec:numerical} we illustrate our results with a few numerical examples. Section \ref{sec:summary} summarizes our results.\\
 Before we proceed with a short review of the studied equations we note a number of works similar in spirit to our own. A priori results concerning compactly supported solutions of (\ref{kmn}) were derived in \cite{ambrosesupport} using completely different methods from the ones to be presented. The first main result in \cite{ambrosesupport} was the conservation of the measure of solution's support (does not preclude its motion). Thus its subject matter is similar in part to our Lemma \ref{support}. Yet, it cannot substitute our Lemma \ref{support}, for it does not imply the confinement of evolution within the initial support and thus cannot be used for the proof of our main result, Theorem \ref{finitetime}. The second main result of \cite{ambrosesupport} was preservation of nonnegativity which is thus akin to our Lemma \ref{positivity}. Here we note that, unlike \cite{ambrosesupport}, our Lemma \ref{positivity} (which generalizes a result derived in \cite{k22finitedifference}) is not limited to the periodic case and our assumptions of regularity are laxer.\\
 With our results being a priori, it should be noted that very little is known about existence of solutions to degenerate dispersive equations. For strictly positive solutions, a general existence theory was derived in \cite{degendisp}. A special case (not of the $K(m,n)$ class) where existence of strong solutions could be proved for solutions which do touch the $x$-axis is found in \cite{ambrosesupport}. Unfortunately, these results have no bearing our own subject matter, compactly supported solutions of the $K(m,n)$ equations, where existence of solutions remains a completely open issue. In \cite{illposed}, a different aspect of well posedness was studied with numerical evidence which counterpoints a continuous dependence in the $H^2$ norm of $K(2,2)$ solutions on the initial condition.
\subsection{Certain features of the ${\cal C}_1(m,a,b)$ equations}
For future reference we start summarizing certain basic features of our problem \cite{Rosenau03,stationary}.
The ${\cal C}_1(m,a,b)$ equations admit {\it both traveling and stationary compactons}. The parameter
 $$\omega:=b+1-a$$
has a crucial impact on the dynamics; both numerical simulations and formal analysis have confirmed that for $\omega>0$, {\it traveling compactons} are evolutionary (note that  for $K(m,n)$, $\omega=n+1$), whereas if $\omega<0$, {\it stationary compactons} are evolutionary and emerge out of compact initial excitations \cite{Rosenau03,stationary}.
Thus, for instance, whereas the ${\cal C}_1(2,1,1)$ equation ($\omega=1$) supports evolutionary traveling compactons
\begin{equation}\label{c211c}
u(x,t)=2\lambda\cos^2\left(\frac{x-\lambda t}2\right)H\left(\pi-|x-\lambda t|\right),
\end{equation}
 in the ${\cal C}_1(4,3,1)$ case wherein $\omega=-1$, the stationary compactons
\begin{equation}\label{c431c}
u(x,t)=u_0\cos(x)H\left(\frac\pi 2-|x|\right),
\end{equation}
are the ones to emerge from a compact initial excitation. As to the regularity of
 the dispersive part of ${\cal C}_1(4,3,1)$ at compacton's edge, we note that since $\left[u^3u_{xx}\right]_x=\left[\frac 14(u^4)_{xx}-3u^2u_x^2\right]_x$, it is clear that all terms are well defined at the singularity.\vspace{5mm}\\
We shall focus on $\omega\geq2$ cases making the traveling compactons the main object of our interest. Formal analysis of the traveling wave equation of the \cmab equation \cite{stationary} shows that, if a compacton solution $u^c$ is permissible, then at the point $x_c$ where the trough of the underlying traveling wave is glued to the zero state at $x\geq x_c$ (respectively $x\leq x_c$) we have
\begin{equation}\label{compsing}
u^c(x,t)\sim(x-x_c)^{\frac2{n-1}}H(x-x_c) \textrm{ as } x\rightarrow x_c.
\end{equation}
 Since we assume that $n\geq2$, then for traveling compactons
$u^c_{3x}\notin L^\infty$.
\vspace{5mm}
\\
Note that the \cmab equations admit the two conserved quantities:
\begin{equation}
I_1=\int udx,\:I_\omega=\frac 1{\omega}\int u^{\omega}dx.
\end{equation}
For certain subcases of the \cmab equations, additional conservation laws are available \cite{hamilcmab}.
\section{Main results}\label{sec:estimates}
We consider the following initial value problem for the ${\cal C}_1(m,a,b)$, $m,n\geq2$ equations
\begin{subnumcases}{\label{Cp}}
  u_{t} +(u^m)_{x} + \frac{1}{b}[u^{a}(u^{b})_{xx}]_{x}=0, & $t\in(0,T],x\in X$,\label{ppde}
  \\
        u(x,0)=u_0(x), & $x\in X$,\label{pinit}
\end{subnumcases}
where either $X=\mathbb{R}$ or $X=[0,L]$. In the latter case (\ref{Cp}) is appended with periodic boundary conditions.
\begin{definition}\label{def}
A strong solution of the initial value problem (\ref{Cp}) is a function $u\in C([0,T]:L^1(X))\cap C([0,T]:W^{3,\infty}(X))$ such that
\begin{equation*}
u_t(\cdot,t)\in L^1(X),\; \forall t\in(0,T]
\end{equation*}
which satisfies (\ref{pinit}) and the PDE (\ref{ppde}) a.e. in $X\times(0,T]$.
\end{definition}
At times we will focus on the $K(m,n)$, $m,n\geq2$ subcase of (\ref{Cp})
\begin{subnumcases}{\label{kmp}}
  u_{t} +(u^m)_{x} + (u^{n})_{3x}=0, & $t\in(0,T],x\in X$\label{kppde}
  \\
        u(x,0)=u_0(x), & $x\in X$\label{kpinit}
\end{subnumcases}
 keeping the same definition of the strong solution. Note that since $n\geq2$ is assumed, by (\ref{compsing}) the traveling compactons are not strong solutions.
\subsection{A priori estimates}
We now proceed to prove a Lemma which ensures that strong solutions of the initial value problem (\ref{Cp}) cannot escape the initial support of $u_0$. Thus the Lemma implies a waiting time property of strong solutions of the \cmab equations, already noted by us in \cite{stationary}. This is the first such rigorous result within the realm of dispersive equations.
\begin{lemma}\label{support}
Let $u(x,t)$ be a strong solution of (\ref{Cp}).
Assume there is a closed interval $I$  such that $u_0(x)=0,\forall x\in I$. Then $u(x,t)=0,\forall x\in I,t\in[0,T]$.
\end{lemma}
\begin{proof}
$u(x,t)$ satisfies
\begin{equation}\label{c1nc}
u_t+A_1u^{m-1}u_x+A_2u^{n-1}u_{3x}+A_3u^{n-2}u_xu_{2x}+A_4u^{n-3}u_x^3=0
\end{equation}
a.e. in $X\times(0,T]$ for constants $A_i$ depending on $m,a,b$.
Multiply (\ref{c1nc}) by $u$ and integrate over $I$
\begin{equation}
\begin{split}
\frac 12\frac d{dt}\intii{u^2}=-\Bigg[A_1&\intii{u^mu_x}+A_2\intii{u^{n}u_{3x}}\\
&+A_3\intii{u^{n-1}u_xu_{2x}}+A_4\intii{u^{n-2}u_x^3}\Bigg].
\end{split}
\end{equation}
By assumptions on solution's regularity; $\forall t\in[0,T], u(x,t)\in C^2(I)$. We have ($C_i=\textrm{consts},i=1,..,9$)
\begin{equation}
\intii{|u^mu_x|}\leq C_1\intii{u^2}\quad\textrm{ and }\quad
\intii{|u^nu_{3x}|}\leq C_2\intii{u^2}.\\
\end{equation}
Noting the following Gagliardo Nirenberg inequalities
\begin{align*}
\norm{u_x}{3}{I}\leq C_3\norm{u_{3x}}{\infty}{I}^{1/3}\norm{u}{2}{I}^{2/3}+C_4\norm{u}{2}{I},\\
\norm{u_{2x}}{6}{I}\leq C_5\norm{u_{3x}}{\infty}{I}^{2/3}\norm{u}{2}{I}^{1/3}+C_6\norm{u}{2}{I},
\end{align*}
we obtain
\begin{equation*}
\intii{|u^{n-2}u_x^3|}\leq C_7\norm{u}{2}{I}^2
\end{equation*}
and by H\"{o}lder's inequality we have
\begin{equation*}
\intii{|u^{n-1}u_xu_{2x}|}\leq\norm{u^{n-1}}{2}{I}\norm{u_x}{3}{I}\norm{u_{2x}}{6}{I}\leq C_8\norm{u}{2}{I}^2.
\end{equation*}
Combining the above estimates, we have
\begin{equation}
\frac d{dt}\intii{u^2}\leq C_9\intii{u^2}
\end{equation}
and by Gr\"{o}nwall's Lemma and the continuity of $u$, we have
\[u(x,t)=0,\quad\forall x\in I,t\in[0,T].\]
\end{proof}
Note that we may prove uniqueness of strong solutions of the ${\cal C}_1(m,a,b)$ equations
by following the method of proof of Lemma \ref{support}.
Since numerical simulations suggest the compactons (which are not strong solutions) dominate the dynamics, this uniqueness result is not of much use and we do not further expand on this issue.\vspace{5mm}\\
The following Lemma is a direct generalization (under more stringent assumptions) of a result obtained in \cite{k22finitedifference} (a result similar in spirit also appears in \cite{ambrosesupport}).
\begin{lemma}\label{positivity}
Let $u(x,t)$ be a strong solution of (\ref{Cp}) such that $\omega\geq 2$
in (\ref{ppde}). Then if $u_0(x)$ is nonnegative, $u(x,t)$ remains nonnegative $\forall t\in(0,T]$.
\end{lemma}
\begin{proof}
Let $u_-=H(-u)u$ be the negative part of $u$. Then $u_-\in C\left(X\times[0,T]\right)$ and $u_x(\cdot,t),u_t(\cdot,t)$ may have jump discontinuities. We multiply (\ref{ppde}) by $u_-^{\omega-1}$ and integrate by parts over $X$
\begin{equation}\label{nonnegint}
\intx{\left[u_-^{\omega-1}u_t-\left(u_-^{\omega-1}\right)_x\left(u^m+\frac{1}{b}u^a(u^b)_{xx}\right)\right]}=0.
\end{equation}
The value of the integral in (\ref{nonnegint}) does not change if $u$ is replaced by $u_-$
\begin{equation*}
\intx{\left[\frac 1\omega(u_-^\omega)_t-\frac{\omega-1}{m+\omega-1}(u_-^{m+\omega-1})_x-\frac{\omega-1}{2b^2}\left(\left[(u_-^b)_x\right]^2\right)_x\right]}=0.
\end{equation*}
At a point $a$ where $(u_-(\cdot,t))_x$ is discontinuous, we have
\[[u_-(\cdot,t)]\sim(x-x_c)H(x-x_c) \textrm{ as } x\rightarrow x_c.\]
It follows that \[\left[(u_-^b)_{x}\right]^2\sim(x-x_c)^{2b-2}H(x-x_c) \textrm{ as } x\rightarrow x_c\]
and since $b\geq2$, $\left[(u_-^b(\cdot,t))_{x}\right]^2$ and $u_-^{m+\omega-1}(\cdot,t)$ are $C^1(X)$.
Thus $\frac d{dt}I_\omega(u_-)=0$ and the result follows.
\end{proof}
For $u_0(x)$ with a compact support we now define
\begin{equation*}
x_0=\inf\left\{x:u_0(x)\neq 0\right\},\; x_1=\sup\left\{x:u_0(x)\neq 0\right\},\; d=x_1-x_0.
\end{equation*}
We have established that strong solutions of (\ref{Cp}) with $u_0(x)$ compactly supported and nonnegative must wait at the fronts of $u_0(x)$ at $x_0,x_1$ and remain nonnegative $\forall t\in(0,T]$.
Now we proceed to the next step and prove that strong solutions of the $K(m,n)$ equations which initially are nonnegative and compactly supported, must lose their regularity {\it in a finite time}. As observed numerically, this loss of regularity is a prerequisite for the emergence of compactons.\vspace{5mm}
\begin{theorem}\label{finitetime}
(Loss of regularity) 
Let $u(x,t)$ be a strong solution of (\ref{kmp}). Assume $u_0(x)$ to be nonnegative, compactly supported and nontrivial (in the periodic case; $0<x_0,x_1<L$). Then
\begin{equation}\label{upb1}
T\leq \frac{dI_1(u_0)-\intia{xu_0(x+x_0)}}{I_1(u_0)^md^{1-m}}
\end{equation}
is finite.
\end{theorem}
\begin{proof}
By Lemma \ref{positivity}, $u(x,t)$ is nonnegative for $0\leq t\leq T$.
By Lemma \ref{support}
\begin{equation*}
x_0\leq\inf\left\{x:u(x,t)\neq 0\right\},x_1\geq\sup\left\{x:u(x,t)\neq 0\right\},\quad \forall\; 0\leq t\leq T.
\end{equation*}
Multiply (\ref{kppde}) by $x$ and integrate over $X$ to get
\begin{align*}
0=&\intx{x\left[u_t+(u^m)_x+(u^n)_{3x}\right]}=\inti{x\left[u_t+(u^m)_x+(u^n)_{3x}\right]}=\\
=&\frac d{dt}\inti{xu}+x(u^m+(u^n)_{xx})\Big|^{x_1}_{x_0}-\inti{\left[u^m+(u^n)_{xx}\right]}=\\
=&\frac d{dt}\inti{xu}-\inti{u^m}
\end{align*}
Note that Lemma \ref{support} assures that all boundary terms vanish. Thus
\begin{equation}\label{cetermassspeed}
\frac d{dt}\inti{xu(x,t)}=\inti{|u^m(x,t)|}, \quad \forall\; 0\leq t\leq T.
\end{equation}
By H\"{o}lder's inequality
\begin{equation}\label{speedlb}
\inti{|u^m(x,t)|}\geq I_1(u_0)^md^{1-m},\quad \forall\; 0\leq t\leq T
\end{equation}
where the RHS is a positive constant.
Thus
\[\inti{xu(x,t)}\geq \left[I_1(u_0)^md^{1-m}\right]t+\inti{xu_0(x)}, \quad \forall\; 0\leq t\leq T\]
and
\[\inti{xu(x,t)}\leq x_1I_1(u_0),\quad \forall\; 0\leq t\leq T,\]
so that
\[T\leq \frac{dI_1(u_0)-\intia{xu_0(x+x_0)}}{I_1(u_0)^md^{1-m}}.\]
\end{proof}
For the $K(n+1,n)$ we may derive a better bound:
\begin{corollary}\label{knp1n}
Let $u(x,t)$ be a strong solution of (\ref{kmp}) where $m=n+1$. Assume $u_0(x)$ to be nonnegative, compactly supported and nontrivial (in the periodic case; $0<x_0,x_1<L$). Then
\begin{equation}\label{upb2}
T\leq \frac{dI_1(u_0)-\intia{xu_0(x+x_0)}}{\omega I_\omega(u_0)}
\end{equation}
is finite.
\end{corollary}
\begin{proof}
Since now the RHS of (\ref{cetermassspeed}) is a constant of the motion $\omega I_\omega$, the result follows. By (\ref{speedlb}) this is a better bound.
\end{proof}
An additional bound for the $K(m,n)$ equations may be derived using a similar argument for the third moment
\begin{equation}\label{upb3}
T\leq \frac{x_1^3I_1(u_0)-\inti{x^3u_0(x)}}{6I_1(u_0)^nd^{1-n}}.
\end{equation}
{\bf Discussion:} Theorem \ref{finitetime} gives no information as to the character of the resulting singularity which develops in finite time. It does not preclude the possibility of a blowup, or a formation of a genuine shock.
Yet the numerically observed compacton has a much milder singularity (see section \ref{sec:numerical}).
For the $K(2,2)$ equation the emergence of a compacton means that the second derivative becomes discontinuous and the solution is no longer strong. Still, Theorem \ref{finitetime} gives no information on the eventual dynamics and the role of compactons. The following sheds some light on the eventual $K(n,n)$ dynamics:
\begin{lemma}\label{decompose}
\normalfont{\cite{hamilcmab}} There are compactly supported initial excitations that the flow of the $K(n,n)$ equations do not decompose solely into compactons and anticompactons (cf. solution (\ref{k22c})). These include all initial conditions with a support given on an interval of length $L<\pi$ which are smooth in the interior of their domain.
\end{lemma}
The proof of Lemma \ref{decompose} is presented in \cite{hamilcmab}.
\subsection{General ${\cal C}_{1}(m,a,b)$ equations ($a\neq0$)}
Multiply (\ref{ppde}) by $xu^{b-a}$ and integrate by parts to obtain
\begin{equation}\label{cmabfm}
\frac d{dt}\inti{\frac{xu^\omega(x,t)}\omega}=\frac m{m+b-a}\inti {u^{m+b-a}}-\frac{3b-a}{2b^2}\inti{[(u^b)_x]^2}.
\end{equation}
However, it is not obvious whether the RHS is bounded from below by a constant of the motion. When $\omega=2$ we have an Hamiltonian \cite{Rosenau03}
\begin{equation*}
\mathcal{H}(u)=\int \left[\frac 1{m+1}u^{m+1}-\frac 1{2b^2}\left[\left(u^b\right)_x\right]^2\right]dx,\quad
u_t=-\partial_x\frac{\delta}{\delta u}\mathcal{H}(u)
\end{equation*}
and (\ref{cmabfm}) reduces to
\begin{equation}
\frac d{dt}\inti{\frac{xu^2(x,t)}2}=\frac {m-2a-3}{m+1}\inti {u^{m+1}(x,t)}+(2a+3)\mathcal{H}(u_0)
\end{equation}
which, depending on the sign of $m-2a-3$ and $\mathcal{H}(u_0)$, can be analysed by similar arguments as in the proof of Theorem \ref{finitetime}. In the distinguished $m=2a+3$\cite{stationary} case, for a nonzero initial Hamiltonian we obtain again that the initially smooth excitations must lose their regularity.
\subsection{Additional equations amenable to analysis}
We proceed to specify additional classes of degenerate evolution equations amenable to a similar analysis:
\subsubsection{An extended Burger's equation}
Given an extended Burger's equation
\begin{equation}\label{degburgers}
u_t+(u^m)_x-(u^n)_{xx}=0,\quad m,n\geq2,
\end{equation}
an equivalent of Lemma \ref{support} can be derived replacing $W^{3,\infty}(X))$ with $W^{2,\infty}(X))$ in definition \ref{def} and using appropriate Gagliardo Nirenberg inequalities. The proofs of Lemma \ref{positivity} and Theorem \ref{finitetime} are also essentially the same. However, for eq. (\ref{degburgers}) much stronger results have already been established, see for instance \cite{freeboundaryenergybook}.
\subsubsection{An extended quintic dispersion equation}
For an extended quintic dispersion equation\cite{quintic,quinticwindow,cooperquintic} $\mathcal{Q}(m,a_1,b_1,a_2,b_2)$:
\begin{equation}\label{degquintic}
u_t+(u^m)_x+\left[u^{a_1}(u^{b_1})_{xx}\right]_{x}+\delta\left[u^{a_2}(u^{b_2})_{4x}\right]_{x}=0,\quad 
a_1,a_2\geq0,\;m,b_1,b_2\geq2
\end{equation}
an equivalent of Lemma 1 can be derived replacing $W^{3,\infty}(X))$ with $W^{5,\infty}(X))$ in definition \ref{def} and using appropriate Gagliardo Nirenberg inequalities. The stronger assumption is not only sufficient for the proof's technicalities, but also necessary since explicit {\it traveling} compacton solutions of (\ref{degquintic}) have been found for which $u_{3x}\in L^\infty$ \cite{quintic}.\\
If we focus on the $a_1=a_2,b_1=b_2$ case, then to emulate Lemma 2 a higher nonlinearity is needed, e.g. $b_2+1-a_2\geq4\Rightarrow b_2\geq3$.\\
Thus for $a_1=a_2=0$ and $b_1=b_2\geq3$ it is straightforward to prove an equivalent of Theorem \ref{finitetime}.
\subsubsection{A $n$-Cubic equation}
For a $n$-Cubic equation \cite{bicubic}
\begin{equation}\label{bicub}
v_t+\left[(v+v_{xx})^n\right]_x=0,\quad n\geq2
\end{equation}
an equivalent of Lemma \ref{support} can be proved by assuming in definition \ref{def} that  $v\in W^{5,\infty}(X))$  and using appropriate Gagliardo Nirenberg inequalities. Again the stronger assumption is necessary since (\ref{bicub}) has the traveling solution
\[v^c(x,t)=\frac{16\lambda}9\cos^4\left(\frac{x-\lambda t}{4}\right)H(2\pi-|x-\lambda t|),\quad v^c_{3x}\in L^\infty.\]
However, Lemma \ref{positivity} does not carry over immediately since, except for $p=1$, the $n$-Cubic equation does not conserve $\int v^p(x,t)dx$. \\
Note that $u=v+v_{xx}$ satisfies formally the $K(n,n)$ equation\cite{bicubic}.
Thus, a similar result concerning loss of regularity in finite time of solutions of (\ref{bicub}) may be obtained by assuming sufficient initial regularity such that $u=v+v_{xx}$ satisfies the conditions of Theorem \ref{finitetime} and arguing by contradiction.\\
\subsubsection{The $K(m,n)$ equations for $1<m,n<2$}
Since for traveling compactons $u^c(x,t)\sim(x-x_c)^{\frac2{n-1}}H(x-x_c)$ with $x_c$ being the location of compacton's front, $u^c$ gains smoothness as $n$ decreases. Obviously, Lemma \ref{support} cannot hold in this case, as traveling compactons may have smooth third derivatives. For instance, the $K(3/2,3/2)$ equation admits the traveling compacton
\[u^c(x,t)=\frac{36\lambda^2}{25}\cos^4\left(\frac{x-\lambda t}{6}\right)H(3\pi-|x-\lambda t|),\]
which has continuous third order derivatives.
\\Yet, assuming the boundedness of higher than third order derivatives in Definition 1 one may replicate Lemma \ref{support}.
\section{Numerical examples}\label{sec:numerical}
Though the discussion in section \ref{sec:estimates} is independent of the unsettled issue of existence we have ample numerical evidence which confirm (and indeed have initiated) our a priori results. Simulations of (\ref{kmp}) have been carried using a Local Discontinuous Galerkin method (LDG) \cite{nodaldisbook} and a choice of numerical fluxes based on \cite{levi}.\\
In figure \ref{k22sim} we present the results of numerical study of a periodic initial value problem for the $K(2,2)$ equation with $u_0(x)=\cos^3(x/10)H(5\pi-|x|)$. The initially sufficiently smooth excitation maintains its support until the second derivative becomes discontinuous (at $6<t<8$) and $u_{3x}\notin L^\infty$. At this time the support is breached (and a compacton emerges). The values of the upper bounds on existence time are not very impressive; eq. (\ref{upb1}) yields $T\leq37.01$ and from eq. (\ref{upb3}), $T\leq1522.02$.\\
In figure \ref{k32sim}, the $K(3,2)$ equation is solved using the same initial condition. The story is qualitatively the same. Here the support is breached at $7<t<9$. Whereas Eqs. (\ref{upb1}) and (\ref{upb3}) yield $T\leq87.20$ and $1522.02$ respectively, the improved upper bound (\ref{upb2}) for the $K(n+1,n)$ equations yields a much better estimate $T\leq25.77$.
\FloatBarrier
\begin{figure}[ht]\centering
\begin{center}
 \makebox[\textwidth]{
\subcaptionbox*{}{\includegraphics[scale=0.6]{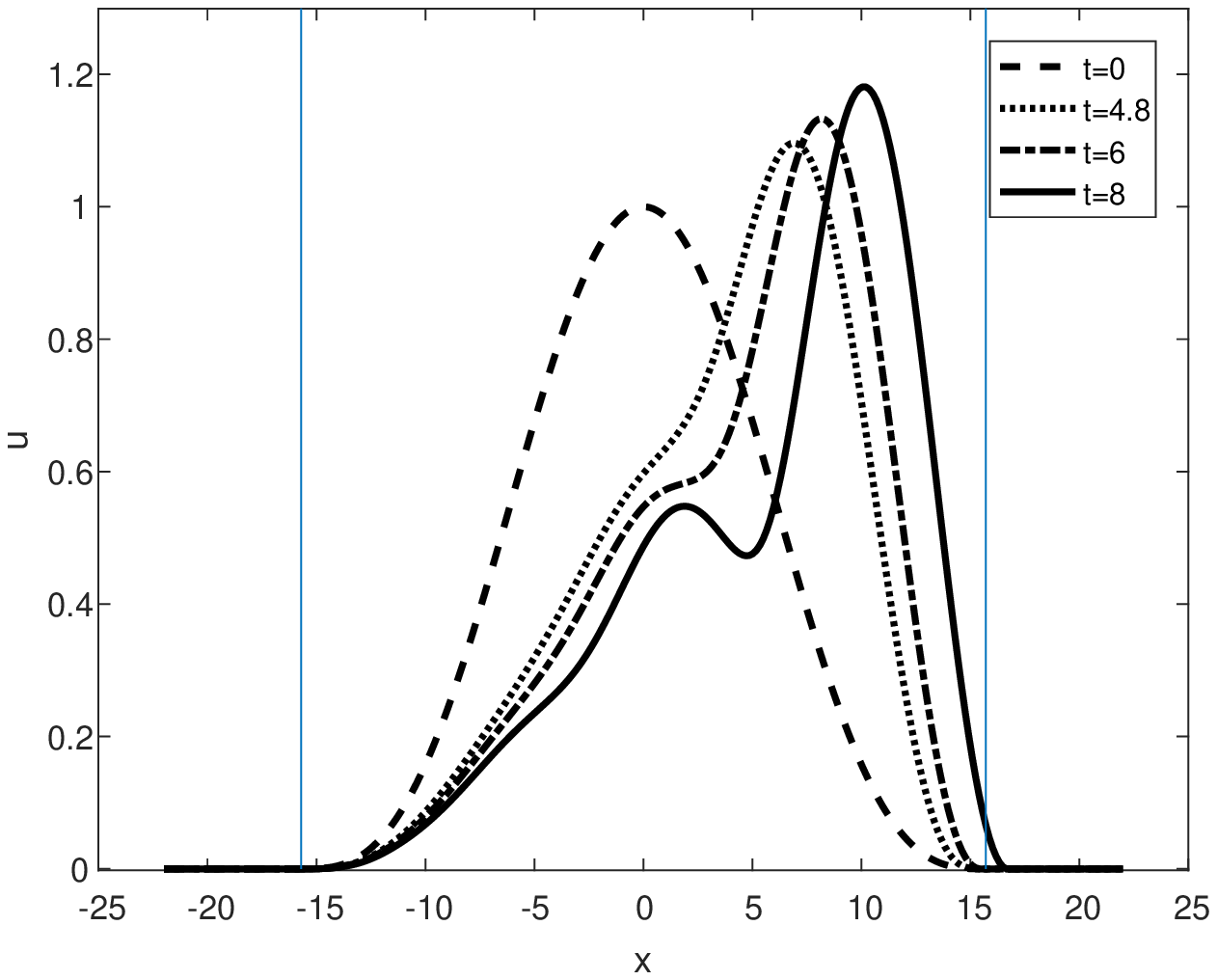}}\hspace*{-0.7cm}
\includegraphics[scale=0.6]{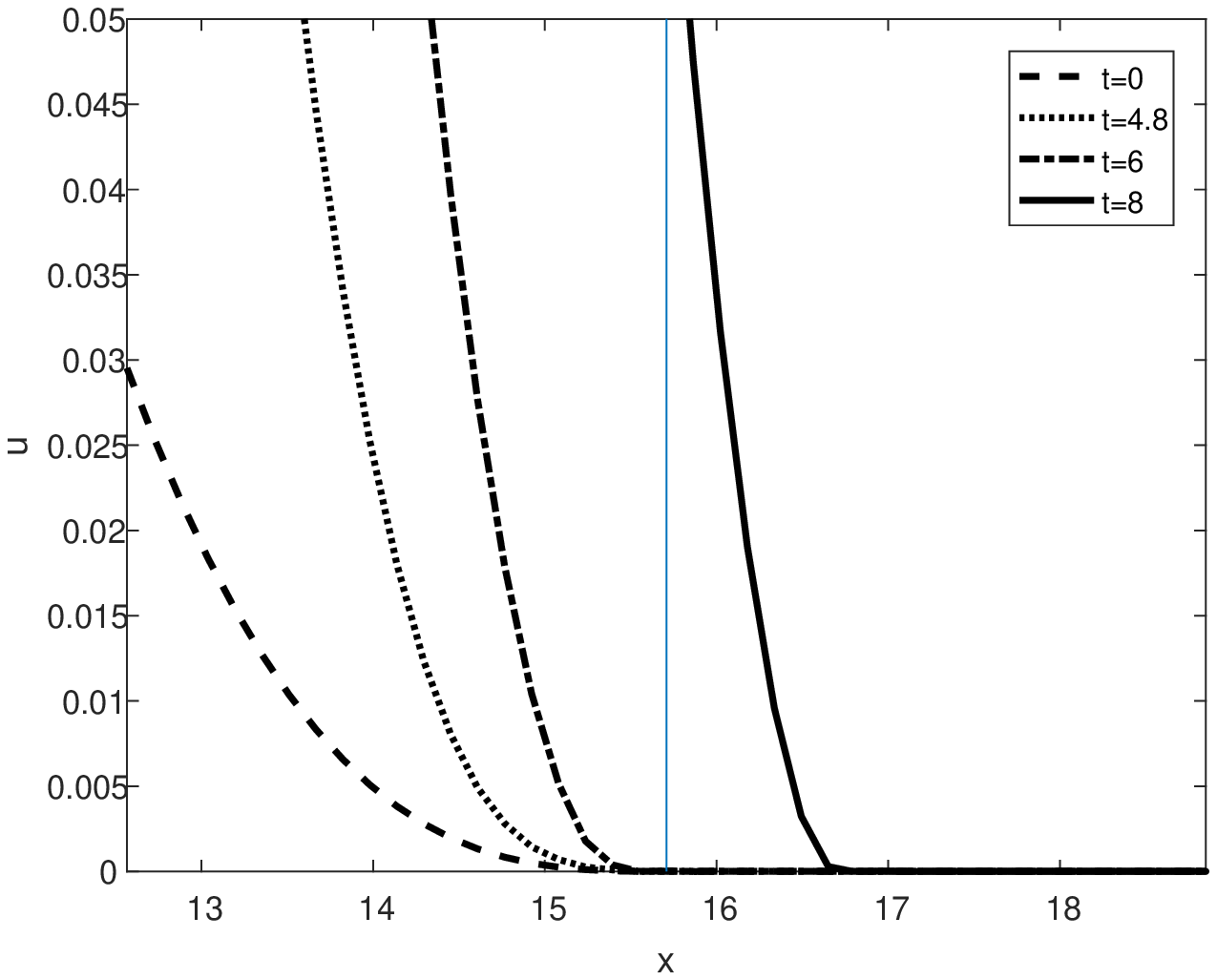}
}
\end{center}
\vspace*{-0.7cm}
\caption{Left plate: simulation of $K(2,2)$ with $u_0(x)=\cos^3(x/10)H(5\pi-|x|)$ and thus $(u_0(x))_{3x}\in L^\infty$. The initial support (vertical lines) is maintained initially while the center of mass is seen to shift to the right. Between $t=6$ and $t=8$ the second spatial derivative becomes discontinuous at the right edge, the barrier is lifted and the solution is allowed to propagate into the vacuum. Right plate: zooming in on the right front.\label{k22sim}}
\end{figure}
\FloatBarrier
\FloatBarrier
\begin{figure}[ht]\centering
\begin{center}
 \makebox[\textwidth]{
\subcaptionbox*{}{\includegraphics[scale=0.6]{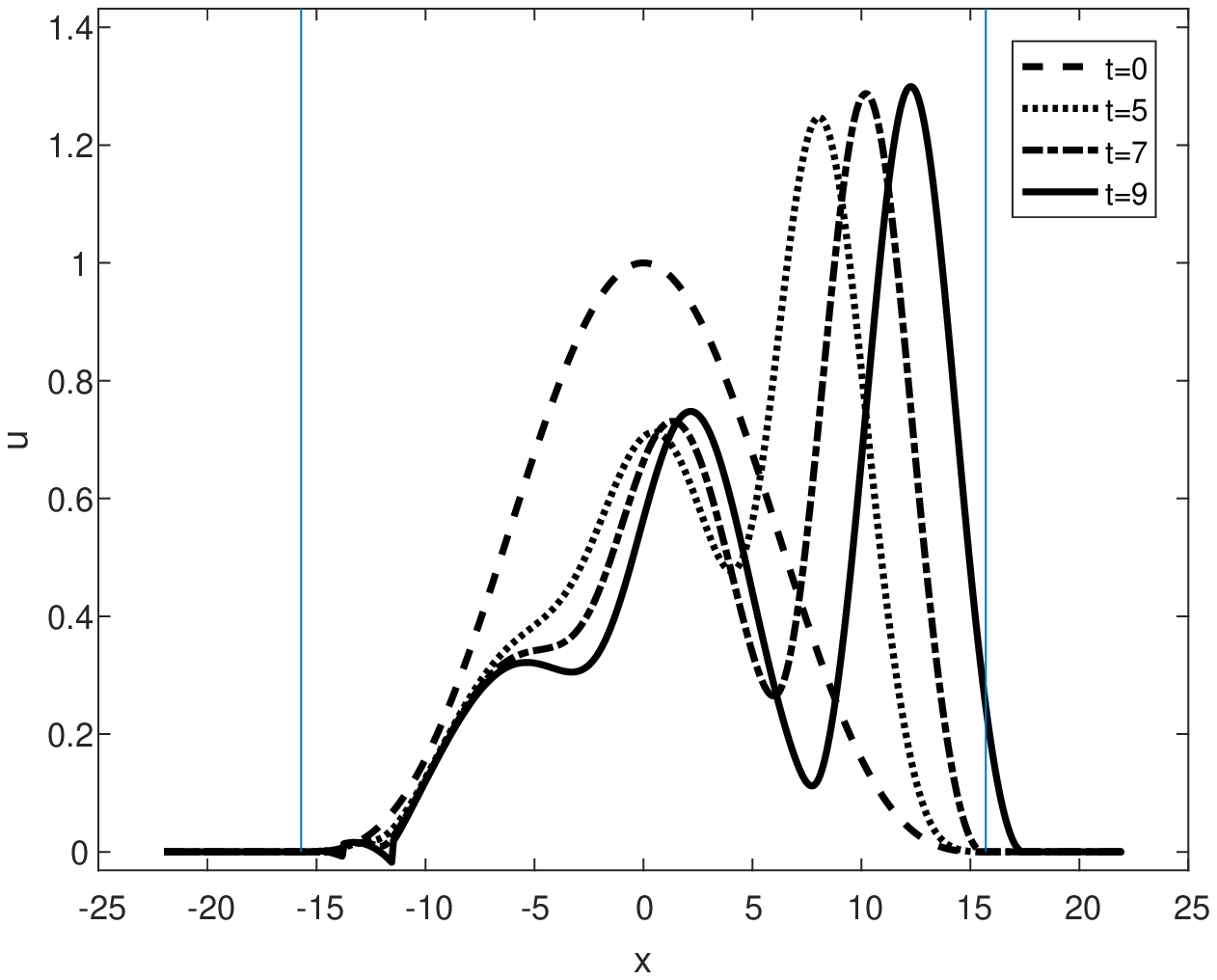}}\hspace*{-0.7cm}
\includegraphics[scale=0.6]{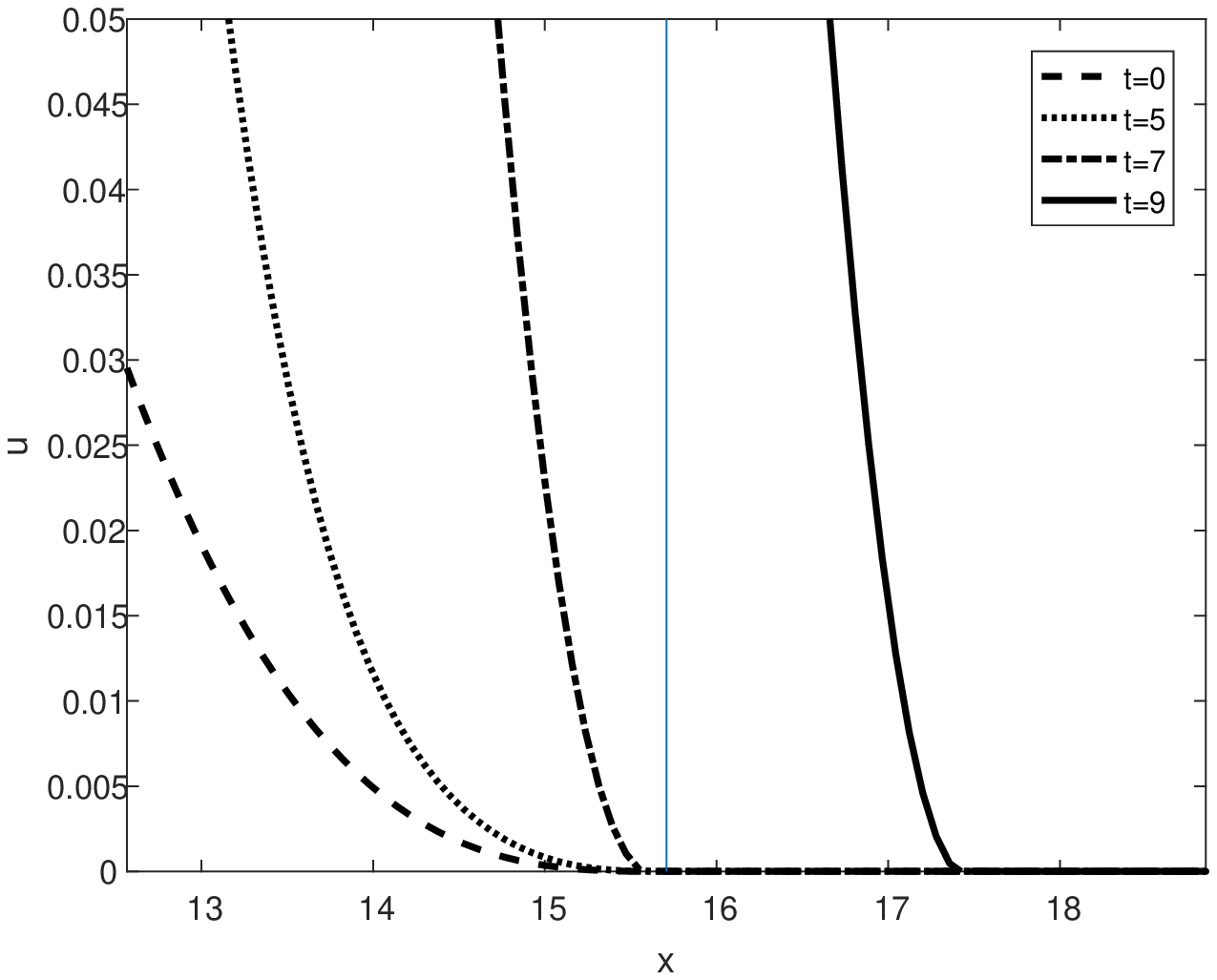}
}
\end{center}
\vspace*{-0.7cm}
\caption{Left plate: simulation of $K(3,2)$ with $u_0(x)=\cos^3(x/10)H(5\pi-|x|)$ and thus $(u_0(x))_{3x}\in L^\infty$. Though the initial support (vertical lines) is initially maintained the center of mass is seen to shift to the right. Somewhere between $7<t<9$ the second spatial derivative becomes discontinuous at the right edge, the barrier is lifted and the solution may propagate into the vacuum. Right plate: zooming in on the right front.\label{k32sim}}
\end{figure}
\FloatBarrier
\section{Summary}\label{sec:summary}
We have proved that initially smooth solutions of the $K(m,n)$ equations must lose their regularity in a finite time. This result extends the already known properties of degenerate evolution equations of lower order. Our results do not concern the actual, more singular, domain of interest of the $K(m,n)$ equations, in which compactons reside. Nevertheless, they support the notion that compactons are in a way attractors of the dynamics, in a similar fashion to classical solitons.
\bibliographystyle{unsrt}
\bibliography{C:/Users/user/Desktop/bibl}
\end{document}